\documentclass[a4paper,12pt]{article}
\usepackage[margin=1in]{geometry}  

\usepackage{graphicx}              
\usepackage{amsmath}
\usepackage{amscd}               
\usepackage{amsfonts} 
\usepackage{amssymb}             
\usepackage{amsthm}                
\usepackage{mathrsfs}
\usepackage{amsbsy}

\numberwithin{equation}{section}

\newtheorem{thm}{Theorem}[section]
\newtheorem{defn}[thm]{Definition}

\newtheorem{prop}[thm]{Proposition}
\newtheorem{cor}[thm]{Corollary}

\newtheorem{question}[thm]{Question}

\DeclareMathOperator{\id}{id}

\makeatletter
\def\th@newremark{\th@remark\thm@headfont{\bfseries}}
\makeatletter

\theoremstyle{newremark}
\newtheorem{rmk}[thm]{Remark}
\newtheorem{eg}[thm]{Example}

\newcommand{\HH}{\mathbb{H}}

\newcommand{\RR}{\mathbb{R}}      

\newcommand{\cC}{\mathcal{C}}

\newcommand{\eE}{\mathcal{E}}

\newcommand{\oO}{\mathcal{O}}
\newcommand{\pP}{\mathcal{P}}

\newcommand{\sS}{\mathcal{S}}

\newcommand{\1}{\textbf{1}}

\begin{document}

\title{Hyperbolic development and inversion of signature}

\author{
Terry J. Lyons\\
\textit{University of Oxford}\\
\and
Weijun Xu\\
\textit{University of Warwick}
}

\maketitle

\abstract{We develop a simple procedure that allows one to explicitly reconstruct any piecewise linear path from its signature. The construction is based on the development of the path onto the hyperbolic space. }

\section{Introduction}
A (Euclidean) path $\gamma$ is a continuous function mapping some finite interval $[0,T]$ into $\RR^{d}$. The length of $\gamma$ is defined as
\begin{align*}
\|\gamma\| := \sup_{\pP \subset [0,T]} \sum_{u_{j} \in \pP} |\gamma_{u_{j+1}} - \gamma_{u_{j}}|, 
\end{align*}
where the supremum is taken over all partitions $\pP$ of the interval $[0,T]$, and
\begin{align*}
|\gamma_{u}| := \bigg( \sum_{j=1}^{d} |\gamma_{u}^{(j)}|^{2} \bigg)^{\frac{1}{2}}
\end{align*}
is the Euclidean norm of the vector $\gamma_{u} = (\gamma_{u}^{(1)}, \dots, \gamma_{u}^{(d)})^{T}$. We say $\gamma$ has \textit{bounded variation} if $\|\gamma\| < +\infty$. 

There are two natural operations on the space of bounded variation paths: concatenation and inverse. For $\alpha: [0,S] \rightarrow \RR^{d}$ and $\beta: [0,T] \rightarrow \RR^{d}$, their concatenation $\alpha * \beta$ is defined as
\begin{equation} \label{eq:concatenation}
\alpha*\beta(u) := \left \{
\begin{array}{rl}
&\alpha(u), u \in [0,S]\\
&\beta(u-S)+\alpha(S)-\alpha(0), u \in [S,S+T]
\end{array} \right.. 
\end{equation}
The inverse of a path $\gamma: [0,T] \rightarrow \RR^{d}$ is defined by $\gamma^{-1}(u) := \gamma(T-u)$. 

We say a path $\gamma$ is irreducible if for every $s < t$, there exists no $u \in (s,t)$ such that $\gamma|_{[s,u]} = (\gamma|_{[u,t]})^{-1}$, where $\gamma|_{[s,u]}$ denotes the segment of $\gamma$ restricted to the time interval $[s,u]$, and similarly for $(\gamma|_{[u,t]})^{-1}$.

If $\gamma$ has bounded variation, then its derivative $\theta(t) = \dot{\gamma}(t)$ exists almost everywhere. We can re-parametrize $\gamma$ in the fixed time interval $[0,1]$ in such a way that
\begin{equation} \label{eq:derivative}
|\theta (t)| := |\dot{\gamma}(t)| \equiv L
\end{equation}
for almost every $t \in [0,1]$, where $|\cdot|$ is the Euclidean norm, and $L$ is the length of $\gamma$. We call such a parametrization the \textit{natural parametrization} of $\gamma$. Note that if $\gamma \in \cC^{1}$ (at natural parametrization), then it is automatically irreducible.

\begin{rmk}
	The notion of natural parametrization defined in \eqref{eq:derivative} is slightly different from the standard one in literature, as we parametrize the path in the unit interval $[0,1]$ rather than $[0,L]$. As a consequence, $\gamma$ has constant speed $L$ instead of $1$. We will see later that it will be convenient for us if we fix the time interval to be $[0,1]$ instead of changing it with length. 
\end{rmk}

For every path of bounded variation, one can associate to it a formal power series whose coefficients are iterated integrals of the path. This formal series is called the \textit{signature} of the path, first introduced by K.T.Chen (\cite{Che54}, \cite{Che57}). Before we give the precise definition of signature, we first introduce a few notations. 

We denote by $\{e_{1}, \dots, e_{d}\}$ the standard basis of $\RR^{d}$. For $n \geq 0$, a word $w$ of length $n$ is a sequence of $n$ basis elements from the set $\{e_{1}, \dots, e_{d}\}$ (with repetition allowed), and we use $|w|$ to denote the length of $w$. For simplicity, we will often write words as sequence of elements from the set $\{1, \dots, d\}$. For example, $w = (2, 3, 1, 1)$ denotes the word $(e_{2}, e_{3}, e_{1}, e_{1})$, and $|w|=4$. We also use $\emptyset$ to denote the empty word, which is the unique word with length $0$. Given a word $w = (i_{1}, \dots, i_{n})$, we let
\begin{align*}
\mathbf{e}_{w} = e_{i_{1}} \otimes \cdots \otimes e_{i_{n}}. 
\end{align*}
With these notations, we now give a precise definition of the signature. 

\begin{defn} \label{def:signature}
	Let $\gamma: [0,T] \rightarrow \RR^{d}$ be a path of bounded variation. For any word $w = (i_{1}, \dots, i_{n})$ with $i_{k} \in \{1, \dots, d\}$, define
	\begin{equation} \label{eq:sig_coefficient}
	C_{\gamma}(w) = \int_{0 < u_{1} < \cdots < u_{n} < T} d \gamma_{u_{1}}^{(i_{1})} \cdots d \gamma_{u_{n}}^{(i_{n})}, 
	\end{equation}
	where $\gamma_{u}^{(i)}$ is the component of $\gamma_{u}$ in the direction of $e_{i}$. The signature of $\gamma$ is then the formal power series
	\begin{equation} \label{eq:signature}
	X(\gamma) = \sum_{n=0}^{+\infty} \sum_{|w|=n} C_{\gamma}(w) \mathbf{e}_{w}, 
	\end{equation}
	where we have set $C_{\gamma}(\emptyset) = 1$. 
\end{defn} 

\begin{rmk} \label{rm:natural_signature}
Note that for any word $w$, \eqref{eq:sig_coefficient} is a definite integral over the time interval where $\gamma$ is defined, so re-parametrizing $\gamma$ does not change its signature. Thus, we can always assume $\gamma$ has the natural parametrization so that \eqref{eq:derivative} holds. 
\end{rmk}

The signature $X(\gamma)$ defined in \eqref{eq:signature} is an element in the tensor algebra $T(\RR^{d}) = \oplus_{n=0}^{+\infty} (\RR^{d})^{\otimes n}$. The map $\gamma \mapsto X(\gamma)$ is then an algebraic homomorphism from the path group (with concatenation and inverse as group multiplication and inverse) into $T(\RR^{d})$ such that
\begin{equation} \label{eq:Chen}
X(\alpha * \beta) = X(\alpha) \otimes X(\beta), \qquad X(\gamma) \otimes X(\gamma^{-1}) = \1, 
\end{equation}
where $\1$ denotes the unit element in $T(\RR^{d})$. The first identity in \eqref{eq:Chen} is known as Chen's identity, first established in \cite{Che57}. 

One reason to look at the signature is that it contains all essential information of a path. In fact, Chen (\cite{Che58}) showed that two irreducible piecewise $\cC^{1}$ paths cannot have the same signatures unless they just differ by a translation or re-parametrization. This result was extended to all bounded variation paths in \cite{HL10}, where the authors showed that two bounded variation paths $\alpha, \beta$ have the same signatures if and only if $\alpha * \beta^{-1}$ is tree-like\protect\footnote{Loosely speaking, a tree-like path has the same effect as a ``zero" path as a control. Since we will make no use of this notion in the article, we skip the precise definition here. Interested readers can find the definition in \cite[Def.~1.2]{HL10} and more background material on trees and paths in \cite{HL08}. 
}. Thus, a very natural question is how one can reconstruct a path of bounded variation from its signature, which is of completely new nature from the purely uniqueness statement in \cite{HL10}. In addition, since there has been an increasing literature concerning the signature-based algorithms (see \cite{LLN13} for its application in time series and \cite{LNO14} for high-frequency financial data), it would be practically useful to develop an effective inversion scheme from signatures to paths. 

In the recent article \cite{LX14}, we developed a procedure based on symmetrization that enables one to reconstruct any $\cC^{1}$ path from its signature. We expect that the same procedure works as long as the path is irreducible and piecewise $\cC^{1}$, whose proof may require slightly more careful treatment than those presented in \cite{LX14}. 

The main goal of this article, however, is to develop an independent procedure to reconstruct any piecewise linear path from its signature. The key ingredient of the procedure is the development of a path onto the hyperbolic space, which was first introduced in \cite{HL10} to recover the length of a path from its signature. We extend this construction to enable us also to recover the direction at the end point of a regular path, which in turn leads to an inversion scheme for piecewise linear paths. Even though the symmetrization method developed in \cite{LX14} covers a more general class of paths than those treated in this article, we believe the approach presented here still has the advantage of being simpler, and the recovery of the derivative at the end point is interesting in its own right. Our main result can be loosely stated as follows. 

\begin{thm}
	Let $\gamma:[0,1] \rightarrow \RR^{d}$ be a path of bounded variation. Let $\theta(t) = \dot{\gamma}(t)$, and let $(\eta_{\lambda}(t), \rho_{\lambda}(t))$ denote the development of the rescaled path $\gamma_{\lambda} = \lambda \gamma$ onto the hyperboloid (see Section \ref{sec:hyperbolic} for details of the construction). Then, for each $\lambda$, the end point $(\eta_{\lambda}(1), \rho_{\lambda}(1))$ of the ``hyperbolic path" can be written explicitly in terms of $X(\gamma)$. If $\gamma \in \cC^{2}$ at natural parametrization, then we have
	\begin{equation} \label{eq:intro_length}
	\theta(1) = \lim_{\lambda \rightarrow +\infty} \eta_{\lambda}(1). 
	\end{equation}
	If $\gamma$ is piecewise linear and its last linear piece has direction $\theta$ and length $l$, then
	\begin{align*}
	\lim_{\lambda \rightarrow +\infty} \frac{1}{\lambda} \log |\eta_{\lambda}(1) - \theta| = -l, 
	\end{align*}	
	that is, both the direction and the length of the last linear piece can be recovered from $X(\gamma)$. 
\end{thm}

Precise statements of the above results are in Theorems \ref{th:recovery_derivative} and \ref{th:piecewise_linear}. This result leads to an inversion procedure for piecewise linear paths as follows. Since we know the complete information of the last linear piece (direction and length) from the signature, Chen's identity \eqref{eq:Chen} allows us to remove that piece from $X(\gamma)$, and thus we are left with the signature of the remaining path segment. One then repeats this procedure until the whole path is recovered.

\begin{flushleft}
	\textbf{Notation and assumptions}
\end{flushleft}
We now summarize the assumptions and notations we will be using in this article. Throughout, we assume the path $\gamma: [0,1] \rightarrow \RR^{d}$ has finite length and is at natural parametrization, so that \eqref{eq:derivative} always holds. Whenever we say $\gamma \in \cC^{k}$, it should be understood that it is $\cC^{k}$ at natural parametrization. 

For any word $w$, we let $C_{\gamma}(w)$ denote the coefficient of $w$ in the signature of $\gamma$ as in \eqref{eq:sig_coefficient}, and use $X(\gamma)$ to denote the whole signature sequence as in \eqref{eq:signature}. In most of this article, we fix the path $\gamma$ under consideration. In that case, we will omit $\gamma$ and simply write $C(w)$ and $X$ instead. Finally, $c$ and $C$ will denote generic constants whose values may change from line to line.

\begin{flushleft}
	\textbf{Organization of the article}
\end{flushleft}
This article is organized as follows. In Section 2, we give an exact inversion procedure for axis paths (paths that move parallel to Euclidean axes). This result is independent of the rest of the article. In Section 3, we show how one can recover the end point derivative of a $\cC^{2}$ path from its signature by developing the path onto the hyperbolic space. Finally, in Section 4, we show how to recover the end point derivative as well as the length of the last piece in the case of piecewise linear paths, thus completing the inversion procedure for such paths. Note that although the main result in Section 4 (Theorem \ref{th:piecewise_linear}) relies on the construction in Section \ref{sec:hyperbolic} (development onto hyperbolic space), its proof is independent of Theorem \ref{th:recovery_derivative}. 

\begin{flushleft}
\textbf{Acknowledgements}
\end{flushleft}
The research of Terry Lyons is supported by EPSRC grant EP/H000100/1 and the European Research Council under the European Union’s Seventh Framework Program (FP7-IDEAS-ERC) / ERC grant agreement nr. 291244. Terry Lyons acknowledges the support of the Oxford-Man Institute. Weijun Xu has been supported by the Oxford-Man Institute through a scholarship during his time as a student at Oxford. He is now supported by Leverhulme trust.

\section{Axis paths}

A path $\gamma$ is a (finite) axis path if it has the form
\begin{align} \label{axis path}
\gamma = (r_{1}e_{i_{1}}) * \cdots * (r_{n}e_{i_{n}}), 
\end{align}
where the $e_{i_{k}}$'s are standard Euclidean basis elements, and ``$*$" is the concatenation defined in \eqref{eq:concatenation}. In other words, an axis path only moves parallel to Euclidean axes, and each piece has finite length. Here, the $r_{k}$'s can be arbitrary non-zero real numbers. If $r_{k}$ is negative, it means that the path moves along the negative direction of $e_{i_{k}}$ for distance $|r_{k}|$. If two consecutive pieces have the same directions up to the sign (that is, $i_{k} = i_{k+1}$), then we can combine them together into one single piece and set $r_{k}' = r_{k} + r_{k+1}$. Thus, we can always assume without loss of generality that $i_{k} \neq i_{k+1}$ for all $k = 1, \dots, n-1$. 

The following notion of square free words has an important role in the characterization of axis paths.

\begin{defn}
A word $w = (i_{1}, \dots, i_{n})$ is square free if $i_{k} \neq i_{k+1}$ for all $k = 1, \dots, n-1$. 
\end{defn}

If the path $\gamma$ is of form \eqref{axis path}, then its signature is
\begin{equation} \label{eq:exponential}
X(\gamma) = \exp(r_{1}e_{i_{1}}) \otimes \cdots \otimes \exp(r_{n}e_{i_{n}}), 
\end{equation}
where each $\exp(r e_k)$ is a formal power series in the basis element $e_k$. Let $w = (i_{1}, \dots, i_{n})$, then by assumption $w$ is square free with $|w| = n$ and
\begin{align*}
C(w) = r_{1} \cdots r_{n} \neq 0. 
\end{align*}
Here, we have written $C(w)$ instead of $C_{\gamma}(w)$ for notational simplicity. Moreover, if $w'$ is any other square free word with $|w'| \geq n$, then we must have $C(w') = 0$ as there is no such term in the expansion of \eqref{eq:exponential}. In other words, for every finite axis path, there is a unique longest square free word $w$ such that $C(w) \neq 0$. Now, given the signature $X$ of some axis path $\gamma$, and suppose $w = (i_{1}, \dots, i_{n})$ is this unique longest square free word, then $\gamma$ necessarily has the form
\begin{align*}
\gamma = (r_{1}e_{i_{1}}) * \cdots *(r_{n}e_{i_{n}}), 
\end{align*}
and all we need to do is to determine the coefficients $r_{k}$'s. To recover these coefficients, for any $k = 1, \dots, n$, we let
\begin{align*}
w_{k} = (i_{1}, \dots, i_{k}, i_{k}, \dots, i_{n}), 
\end{align*}
then it follows from straightforward computations that $C(w_{k}) = \frac{1}{2} r_{1} \cdots r_{k}^{2} \cdots r_{n}$, and so $r_{k} = \frac{2C(w_{k})}{C(w)}$. Thus, we have proved the following reconstruction theorem for axis paths.

\begin{thm}
Let $\gamma$ be a (finite) axis path, and $\{C(w)\}$ be the coefficients in its signature. Then, there is a unique longest square free word $w^{*}$ with $C(w^{*}) \neq 0$. Suppose this word $w^{*}$ has the form $w^{*} = (i_{1}, \dots, i_{k}, \dots, i_{n})$, and denote
\begin{align*}
w^{*}_{k} = (i_{1}, \dots, i_{k}, i_{k}, \dots, i_{n}). 
\end{align*}
Then, we have
\begin{align*}
\gamma = (r_{1}e_{i_{1}}) * \cdots *(r_{n}e_{i_{n}})
\end{align*}
with $r_{k} = \frac{2C(w_{k})}{C(w)}$. 
\end{thm}

As an immediate corollary, we have the following upper bound of the number of terms in the signature needed to reconstruct an axis path.

\begin{cor}
For an axis path with $n$ pieces, one needs at most $n+1$ levels in the signature to reconstruct the path. 
\end{cor}

It is also natural to ask that given an integer $n$, whether there exist two different paths that have the same signatures up to level $n$. We now answer this question in affirmative by an explicit construction. Let $x,y$ be the standard basis elements of $\RR^{2}$, and let $\alpha^{0}, \beta^{0}$ be two $1$-step lattice paths, in $x$ and $y$ directions, respectively. Suppose we have now constructed $\alpha^{n}$ and $\beta^{n}$; we define $\alpha^{n+1}$ and $\beta^{n+1}$ by
\begin{align} \label{recursive relation}
\alpha^{n+1} = \alpha^{n} * \beta^{n}, \qquad \beta^{n+1} = \beta^{n} * \alpha^{n}. 
\end{align}
Then, for each $n$, both $\alpha^{n}$ and $\beta^{n}$ have $2^{n}$ steps, and they are different in all the steps. We now claim that $\alpha^{n}$ and $\beta^{n}$ have the same signatures up to level $n$.

\begin{prop}
For every $n$, we have
\begin{align*}
X^{k}(\alpha^{n}) = X^{k}(\beta^{n}), \phantom{1} \forall k \leq n. 
\end{align*}
\begin{proof}
This is clearly true for $n = 0$. Suppose the proposition holds true for $m = 0, \dots, n$, then for $m = n+1$ and $k \leq n$, using the recursive relation \eqref{recursive relation} and Chen's identity \eqref{eq:Chen}, we have
\begin{align*}
X^{k}(\alpha^{n+1}) = \sum_{j=0}^{k} X^{j}(\alpha^{n}) \otimes X^{k-j}(\beta^{n}) = \sum_{j=0}^{k} X^{j}(\beta^{n}) \otimes X^{k-j}(\alpha^{n}) = X^{k}(\beta^{n+1}), 
\end{align*}
where for the second equality above, we have used the induction hypothesis to switch $\alpha^{n}$ and $\beta^{n}$. For $k = n + 1$, we have
\begin{align*}
X^{n+1}(\alpha^{n+1}) &= X^{n+1}(\alpha^{n}) + X^{n+1}(\beta^{n}) + \sum_{j=1}^{n} X^{j}(\alpha^{n}) \otimes X^{n+1-j}(\beta^{n}) \\
&= X^{n+1} (\beta^{n+1}), 
\end{align*}
thus proving the proposition. 
\end{proof}
\end{prop}

It is easy to see that the path $\alpha^{n} * (\beta^{n})^{-1}$ above has trivial signature in the first $n$ levels, and it has length $2^{n+1}$. An interesting question is the following.

\begin{question}
Can one find a nontrivial lattice path with length shorter than $2^{n+1}$ and that the first $n$ levels in its signature are all zero? 
\end{question}

\section{The derivative at the end point}

In this section, we show how one can approximate the derivative at the end point of a $\cC^{2}$ path through a limiting process of its signature. The main strategy, based on \cite[Sec.~3]{HL10}, is to develop the rescaled path onto a hyperbolic space, and employ the negative curvature to extract information.

\subsection{Development onto the hyperbolic space} \label{sec:hyperbolic}

We first give a brief description of the hyperbolic development of a $\cC^{2}$ path $\gamma$. and then derive the precise system of differential equations (\eqref{eq:hyperbolic_trajectory} below) for the developed path. 

Consider the quadratic form on $\RR^{d+1}$ defined by
\begin{align*}
I(x,y) = \sum_{j=1}^{d} x_{j} y_{j} - x_{d+1}y_{d+1}, 
\end{align*}
and the surface
\begin{align*}
\HH^{d} = \{x \in \RR^{d+1}: I(x,x) = -1, x_{d+1} > 0\}. 
\end{align*}
Here, $x_{j}$'s and $y_{j}$'s are coordinates of $x,y \in \RR^{d+1}$. For any $x \in \HH^{d}$, $I$ is symmetric and positive definite on the tangent space $\{y: I(y,x) = 0\}$, so it gives a Riemannian structure on $\HH^{d}$. In fact, $\HH^{d}$ is the standard upper half $d$-dimensional hyperboloid with metric obtained by restricting $I$ to its tangent spaces. 

Let $SO(d)$ denote the group of orientation preserving isometries on $\HH^{d}$, then its Lie algebra $so(d)$ is the set of $(d+1) \times (d+1)$ matrices of the form
\begin{align*}
\begin{pmatrix}  A & \beta \\ \beta^{T} & 0  \end{pmatrix}, 
\end{align*}
where $A$ is a $d \times d$ anti-symmetric matrix and $\beta \in \RR^{d}$. Now, let $\gamma: [0,1] \rightarrow \RR^{d}$ be a bounded variation path. For the linear operator $F: \RR^{d} \rightarrow so(d)$ defined by
\begin{align*}
F: x \mapsto \begin{pmatrix}  0 & x \\ x^{T}& 0  \end{pmatrix}, \qquad x \in \RR^{d}, 
\end{align*}
$F(\gamma(t))$ is a path in $so(d)$. For any $s < t$, the linear differential equation of the Cartan development $\Gamma_{s,t}$ of the path segment $\gamma|_{[s,t]}$ onto $SO(d)$ is
\begin{align} \label{eq:Cartan_development}
d \Gamma_{s,t}(u) = F(d \gamma(u)) \Gamma_{s,t}(u), \quad u \in [s,t], \quad \Gamma_{s,t}(s) = \id. 
\end{align}
It then follows that for any fixed $s<t$, $\Gamma_{s,t}(u)$ is a path in $SO(d)$, and we have the multiplicative property
\begin{align} \label{multiplicative Cartan}
\Gamma_{u,t}(t) \Gamma_{s,u}(u) = \Gamma_{s,t}(t), \qquad \forall s \leq u \leq t. 
\end{align}
Now, for $t \in [0,1]$, we write $\Gamma(t) = \Gamma_{0,1}(t)$. Since \eqref{eq:Cartan_development} is linear, we have
\begin{align} \label{expansion of the solution in terms of signature}
\Gamma(t) = I + S^{1}_{0,t}(\gamma) + \cdots + S^{n}_{0,t}(\gamma) + \cdots, 
\end{align}
where $I$ is the $(d+1) \times (d+1)$ identity matrix, and
\begin{align*}
S^{n}_{0,t}(\gamma) = \int_{s < u_{1} < \cdots < u_{n} < t} F(d\gamma(u_{n})) \cdots F(d\gamma(u_{1}))
\end{align*}
can be explicitly expressed in terms of the signature of $\gamma$. Note that the right hand side of \eqref{expansion of the solution in terms of signature} is absolutely convergent since $S^{n}_{0,t}(\gamma)$ is an integral over a simplex, which decays factorially in $n$. Let $o = (0, \dots, 0, 1)^{T} \in \RR^{d+1}$, then $\Gamma(t)o$ is a path on the hyperboloid $\HH^{d}$, and we call it the development of $\gamma$ on $\HH^{d}$. 

It is standard that every point on $\HH^{d}$ can be uniquely expressed as
\begin{align*}
\begin{pmatrix} \eta \sinh \rho \\ \cosh \rho \end{pmatrix}
\end{align*}
for some $\eta \in \mathbb{S}^{d-1}$ and $\rho \in \RR^{+}$ (see for example \cite{CFKP97}). For any $\lambda > 0$, we let $\gamma_{\lambda} = \lambda \gamma$ be the rescaled path, and write
\begin{equation} \label{eq:development_hyperboloid}
\Gamma_{\lambda}(t) o = \begin{pmatrix}  \eta_{\lambda}(t) \sinh \rho_{\lambda}(t) \\ \cosh \rho_{\lambda}(t) \end{pmatrix}
\end{equation}
as the development of $\gamma_{\lambda}$ on $\HH^{d}$.

\begin{rmk}
	We first give an explicit expression of the end point of the development of $\gamma_{\lambda}$ on $\HH^{d}$ in terms of the signature of $\gamma$. For each $n \geq 0$, let
	\begin{align*}
	\eE_{2n} = \big\{w: w = (i_{1}, i_{1}, \dots, i_{k}, i_{k}, \dots, i_{n}, i_{n}), i_{j} = 1, \cdots, d \big\}. 
	\end{align*}
	Also, for $k = 1, \dots, d$, let
	\begin{align*}
	\eE_{2n}^{(k)} = \{w: w = (\tilde{w}, k), \tilde{w} \in \sS_{2n} \}. 
	\end{align*}
	Using the expression \eqref{expansion of the solution in terms of signature} and then \eqref{eq:development_hyperboloid}, we can compute
	\begin{equation} \label{eq:expression_ep}
	\cosh \rho_{\lambda}(1) = \sum_{n \geq 0} \sum_{w \in \eE_{2n}} C_{\gamma_{\lambda}}(w), \qquad \eta_{\lambda}^{(k)} \sinh \rho_{\lambda}(1) = \sum_{n \geq 0} \sum_{w \in \eE_{2n}^{(k)}} C_{\gamma_{\lambda}}(w), 
	\end{equation}
	where $C_{\gamma_{\lambda}}(\emptyset) = 1$, and $\eta_{\lambda}^{(k)}$ denotes the $k$-th component of $\eta_{\lambda} \in \mathbb{S}^{d-1}$. Since for $|w|=n$ one has $C_{\gamma_{\lambda}}(w) = \lambda^{n} C_{\gamma}(w)$, \eqref{eq:expression_ep} gives the explicit expression of the end point of the ``hyperbolic path" of $\gamma_{\lambda}$ in terms of the signature of $\gamma$. 
\end{rmk}

Note that until this point, the only assumption we have used is that $\gamma$ has finite length (the notation above also assumes $\gamma$ is parametrized on $[0,1]$). However, if we assume more regularity on $\gamma$, then we can infer some important information of $\gamma$ from the behaviour of $(\eta_{\lambda}(1), \rho_{\lambda}(1))$. In fact, it was shown in \cite[Prop.~3.8]{HL10} that if $\gamma$ is $\cC^{2}$ (at natural parametrization) and has length $L$, then there exists $C > 0$ such that
\begin{equation} \label{eq:recovery_length}
\lambda L - \frac{C}{\lambda} \leq \rho_{\lambda}(1) \leq \lambda L
\end{equation}
for all large $\lambda$. Combining \eqref{eq:expression_ep} and \eqref{eq:recovery_length}, one can recover the length $L$ of a $\cC^{2}$ path from its signature by the quantity $\frac{\rho_{\lambda}(1)}{\lambda}$ for large $\lambda$. In order to recover the end direction $\theta(1)$ as well, we need to track the path $\Gamma_{\lambda}(t) o$ on $\HH^{d}$. Before we proceed, we first give an example to illustrate how the hyperbolic development can help the recovery of length and end point direction. 

\begin{eg} \label{eg:lattice}
Let $\gamma = (L_{1}x) * (L_{2}y)$ be a $2$-dimensional axis path, moving along $x$ direction for distance $L_{1}$ first, and then $y$ direction for distance $L_{2}$. Then by \eqref{eq:Cartan_development}, the Cartan development $\Gamma_{\lambda}$ of the whole rescaled path $\gamma_{\lambda}$ onto $SO(d)$ is
	\begin{align*}
	\Gamma_{\lambda} = \begin{pmatrix} 1 &  0 & 0  \\  0 & \cosh \lambda L_{2} & \sinh \lambda L_{2} \\ 0 & \sinh \lambda L_{2} & \cosh \lambda L_{2}  \end{pmatrix} \cdot \begin{pmatrix} \cosh \lambda L_{1} &  0 & \sinh \lambda L_{1}  \\  0 & 1 & 0 \\ \sinh \lambda L_{1} & 0 & \cosh \lambda L_{1}  \end{pmatrix}. 
	\end{align*}
	Multiplying the base point (of $\HH^{2}$) $(0,0,1)^{T} \in \RR^{3}$ with $\Gamma_{\lambda}$, we see the end point of the development of $\gamma_{\lambda}$ on $\HH^{2}$ is
	\begin{align*}
	\begin{pmatrix} \eta_{\lambda}^{(1)} \sinh \rho_{\lambda} \\ \eta_{\lambda}^{(2)} \sinh \rho_{\lambda} \\ \cosh \rho_{\lambda}  \end{pmatrix} = \begin{pmatrix} \sinh \lambda L_{1} \\ \sinh \lambda L_{2} \cosh \lambda L_{1} \\ \cosh \lambda L_{1} \cosh \lambda L_{2} \end{pmatrix}, 
	\end{align*}
where we have written for simplicity $(\eta_{\lambda}, \rho_{\lambda}) = (\eta_{\lambda}(1), \rho_{\lambda}(1))$, and $\eta_{\lambda}^{(k)}$ is the $k$-th component of $\eta_{\lambda} \in \mathbb{S}^{1}$. 
\end{eg}

We have three observations relating to the previous example. 
\begin{enumerate}
	\item Since $\cosh \rho_{\lambda} = \cosh \lambda L_{1} \cosh \lambda L_{2}$, $\rho_{\lambda}$, the hyperbolic distance of the end point to the base point $o$, satisfies
	\begin{align*}
	\rho_{\lambda} = \lambda (L_{1} + L_{2}) + \oO(1)
	\end{align*}
	for all $\lambda$, and the quantity $\oO(1)$ is \textit{negative} for all large $\lambda$. We then see that $\frac{\rho_{\lambda}}{\lambda}$ is ``almost" the length of $\gamma$ if $\lambda$ is large. Note that this observation is weaker than \eqref{eq:recovery_length}, but it is \textit{not} a contradiction since the path considered in this example is not $\cC^{2}$. 
	
	\item By the previous observation, we have $\eta_{\lambda}^{(1)} = \frac{\sinh \lambda L_{1}}{\sinh \rho_{\lambda}} \rightarrow 0$ as $\lambda \rightarrow +\infty$, and hence $\eta_{\lambda}$, the observed direction on $\HH^{d}$, satisfies 
	\begin{align*}
	(\eta_{\lambda}^{(1)}, \eta_{\lambda}^{(2)}) \rightarrow (0,1)
	\end{align*}
	as $\lambda \rightarrow +\infty$. This reflects the fact that the direction of the second piece of the path is vertically up. 
	
	\item This observation is a quantitative version of the previous one. In fact, we have 
	\begin{align*}
	\eta_{\lambda}^{(1)} = \frac{\sinh \lambda L_{1}}{\sinh \rho_{\lambda}} = e^{- \lambda L_{2}} + \oO(e^{-\lambda L}), \quad \eta_{\lambda}^{(2)} = 1 - \frac{1}{2} e^{- 2 \lambda L_{2}} + \oO(e^{- \lambda L}), 
	\end{align*}
	where $L = L_{1} + L_{2}$. Thus, we have
	\begin{align*}
	\frac{1}{\lambda} \log|\eta_{\lambda} - (0,1)^{T}| \rightarrow - L_{2}
	\end{align*}
	as $\lambda \rightarrow +\infty$, the right hand side being the length of the second linear piece. 
\end{enumerate}

These observations are not coincidences, and they are consequences of the negative curvature of the hyperbolic space. In fact, when the scale $\lambda$ is large, the negative curvature will stretch out the path close to a geodesic, and hence the end point will give asymptotically accurate information of the length. When the path is being ``stretched out", it travels exponentially fast when high on the hyperboloid, and hence later directions tend to dominate earlier directions in an overwhelming way. These explain the first two observations, and it turns out that they hold for general piecewise $\cC^{2}$ paths. Thus, when taking $\lambda$ very large, one expects $\eta_{\lambda}(1)$ to be close to $\theta(1)$, and hence one can recover the tangent direction $\theta(1)$ through the limiting behavior of $\eta_{\lambda}(1)$. 

The third observation is true for piecewise linear paths. As we shall see later, though the difference $|\eta_{\lambda} - \theta|$ is of order $\oO(\frac{1}{\lambda})$ for general $\cC^{2}$ paths, it is exponentially small for piecewise linear paths, with $L_{2}$ in the above example replaced by the length of the last linear piece. This estimate, together with the fact that the direction at the end of the path can be approximated arbitrarily closely by its signature sequence, gives an inversion theorem for signatures of piecewise linear paths. 

We now turn to the recovery of the end point direction for general $\cC^{2}$ paths. From this point, we assume $\gamma \in \cC^{2}$ at natural parametrization. We first need to track the trajectory $(\eta_{\lambda}(t), \rho_{\lambda}(t))$ for the development of $\gamma$ on $\HH^{d}$. If $\theta(t) \equiv \theta$; that is, $\gamma$ is a straight line with length $L$, then $\Gamma(t) = \Gamma_{0,1}(t)$ can be written explicitly as
\begin{equation} \label{eq:Cartan_line}
\Gamma(t) = \begin{pmatrix}  (\cosh(Lt) - 1) \theta \theta^{T} + I  &  \sinh(Lt) \theta    \\   \sinh(Lt) \theta^{T}  &  \cosh (Lt)    \end{pmatrix}, 
\end{equation}
where $I$ denotes the $d \times d$ identity matrix. If $\gamma$ is linear in a small time interval $[t,t+\delta]$, then \eqref{eq:Cartan_line} implies
\begin{equation} \label{eq:perturbation}
\begin{split}
\Gamma_{t,t+\delta}(t+\delta) &= \begin{pmatrix} (\cosh(L \delta) - 1) \theta(t) \theta(t)^{T} + I  &  \sinh(L \delta) \theta(t)    \\   \sinh(L \delta) \theta(t)^{T}  & \cosh (L \delta)  \end{pmatrix} \\
&= \begin{pmatrix}  I & L \theta(t) \delta   \\   L \theta(t)^{T} \delta & 1  \end{pmatrix} + o(\delta), 
\end{split}
\end{equation}
where $o(\delta)$ denotes a quantity that vanishes as $\delta \rightarrow 0$. On the other hand, if $\theta \in \cC^{1}$ (or equivalently, $\gamma \in \cC^{2}$), then
\begin{align*}
\sup_{u \in [t,t+\delta]} |\theta(u) - \theta(t)| < C \delta, 
\end{align*}
so \eqref{eq:perturbation} still holds as long as $\gamma \in \cC^{2}$. Since the multiplicative relation \eqref{multiplicative Cartan} implies
\begin{align*}
\begin{pmatrix}  \eta(t + \delta) \sinh \rho(t + \delta) \\ \cosh \rho(t + \delta) \end{pmatrix} = \Gamma_{t,t+\delta}(t + \delta) \begin{pmatrix}  \eta(t) \sinh \rho(t) \\ \cosh \rho(t) \end{pmatrix}, 
\end{align*}
using \eqref{eq:perturbation}, we deduce for $\gamma \in \cC^{2}$ that
\begin{align} \label{eq:difference_equation}
\left \{
\begin{array}{rl}
& \eta(t+\delta) \sinh \rho(t + \delta) - \eta(t) \sinh \rho(t) = L \theta(t) \cosh \rho(t) \delta + o(\delta) \\
& \cosh \rho(t + \delta) - \cosh \rho(t) = L \theta(t)^{T} \eta(t) \sinh \rho(t) \delta + o(\delta)
\end{array} \right.. 
\end{align}
Now for $\lambda > 0$, replacing $L$ by $\lambda L$ in \eqref{eq:difference_equation}, dividing both sides by $\delta$ and sending $\delta \rightarrow 0$, we deduce that the trajectory $(\eta_{\lambda}, \rho_{\lambda})$ on $\HH^{d}$ satisfies the system of differential equations
\begin{equation} \label{eq:hyperbolic_trajectory}
\left \{
\begin{array}{rl}
&\eta_{\lambda}'(t) = \lambda L \coth \rho_{\lambda}(t) \big( \theta(t) - \eta_{\lambda}(t) \theta(t)^{T} \eta_{\lambda}(t) \big) \\
&\rho_{\lambda}'(t) = \lambda L \theta(t)^{T} \eta_{\lambda}(t) \\
&\eta_{\lambda}(0) = \theta(0), \phantom{1} \rho_{\lambda}(0) = 0
\end{array} \right.,  
\end{equation}
where $I$ is the $d$-dimensional identity matrix.

In the next subsection, we will prove the first two observations made in Example \ref{eg:lattice} for general $\cC^{2}$ paths by analyzing the system \eqref{eq:hyperbolic_trajectory}. 

\begin{rmk}
	Note that by inspecting the derivation above, we can see that the differential equations \eqref{eq:hyperbolic_trajectory} still hold as long as $\gamma \in \cC^{1}$. The main reason to assume one more degree of regularity is to be precise about the convergence rate of $\eta_{\lambda} - \theta$ to $0$, as stated in Theorem \ref{th:recovery_derivative} below. 
\end{rmk}

\subsection{Solving the differential equation}

The goal of this section is to prove the following theorem, which gives a quantitative estimate of $|\eta_{\lambda}(t) - \theta(t)|$ for large $\lambda$. 

\begin{thm} \label{th:recovery_derivative}
Let $\gamma(t) = L \theta(t)$, $t \in [0,1]$, where $\theta$ is a $\cC^{1}$ path on the unit sphere $\mathbb{S}^{d-1}$. Let $(\rho_{\lambda}, \eta_{\lambda})$ denote the hyperbolic development of the rescaled path $\gamma_{\lambda}$ as described in the previous subsection. Then, for any $t>0$, we have
\begin{align*}
\lim_{\lambda \rightarrow +\infty} \lambda L  (\eta_{\lambda}(t) - \theta(t)) = - \big(I + \theta(t) \theta(t)^{T}\big)^{-1} \cdot \theta'(t), 
\end{align*}
where $I$ is the $d$-dimensional identity matrix. 
\end{thm}
\begin{proof}
We can assume $L=1$ without loss of generality. The proof for the general case is the same. Let $f_{\lambda}(t) = \eta_{\lambda}(t) - \theta(t)$, then by \eqref{eq:hyperbolic_trajectory} and using $\theta^{T} \theta \equiv 1$, we deduce that $(f_{\lambda}, \rho_{\lambda})$ satisfies
\begin{equation} \label{eq:equation_f}
\left \{
\begin{array}{rl}
&f_{\lambda}'(t) = -\lambda \cdot \coth \rho_{\lambda}(t) \bigg( I + \theta(t)\theta(t)^{T} + f_{\lambda}(t) \theta(t)^{T} \bigg) f_{\lambda}(t) - \theta'(t) \\
&\rho_{\lambda}'(t) = \lambda (1 + \theta(t)^{T} f_{\lambda}(t)) \\
&f_{\lambda}(0) = 0, \phantom{1} \rho_{\lambda}(0) = 0
\end{array} \right., 
\end{equation}
where $f_{\lambda} \in \RR^{d}, \theta \in \mathbb{S}^{d-1}$ and $\rho_{\lambda} \in \RR^{+}$. In what follows, we will show that when $\lambda$ is large, \eqref{eq:equation_f} is ``close" to a system of linear equations, and deduce the limit of $\lambda f_{\lambda}(t)$ from the comparison with the linear system. The proof consists of four steps.

\begin{flushleft}
\textit{Step 1.} 
\end{flushleft}
We claim that there exists $C > 0$ such that
\begin{align*}
\sup_{t} |f_{\lambda}(t)| < \frac{C}{\lambda}
\end{align*}
for all large enough $\lambda$, and any $C > \left\| \theta' \right\|_{\infty}$ should suffice. In fact, whenever the quantity $|f_{\lambda}(t)|$ reaches the value $\frac{C}{\lambda}$, its magnitude will be forced to decrease. To see this, we compute
\begin{align*}
\frac{d}{dt} |f_{\lambda}(t)|^{2} = 2 \big< f_{\lambda}'(t), f_{\lambda}(t) \big>. 
\end{align*}
Substituting $f_{\lambda}'$ with \eqref{eq:equation_f}, we have
\begin{align*}
\frac{1}{2} \frac{d}{dt} |f_{\lambda}(t)|^{2} = &-\lambda  \cdot \coth \rho_{\lambda}(t) \bigg( \big<(I + \theta(t)\theta(t)^{T})f_{\lambda}(t), f_{\lambda}(t)\big> \\
&+ \big< f_{\lambda}(t) \theta(t)^{T} f_{\lambda}(t), f_{\lambda}(t)\big> \bigg) - \big< \theta'(t), f_{\lambda}(t) \big>. 
\end{align*}
Since $\theta \theta^{T}$ is the projection matrix onto $\theta$, the first term in the parenthesis on the right hand side above is always positive, and is bounded from below by
\begin{align*}
\big<(I + \theta(t)\theta(t)^{T})f_{\lambda}(t), f_{\lambda}(t)\big> &= |f_{\lambda}(t)|^{2} + |\big< \theta(t), f_{\lambda}(t)\big>|^{2} \geq |f_{\lambda}(t)|^{2}. 
\end{align*}
Also, since $|\theta(t)| \equiv 1$, we estimate the second term by
\begin{align*}
| \big< f_{\lambda}(t) \theta(t)^{T} f_{\lambda}(t), f_{\lambda}(t)\big> | \leq |f_{\lambda}(t)|^{3}. 
\end{align*}
Finally the last term satisfies $|\big< \theta'(t), f_{\lambda}(t) \big>| \leq \left\| \theta' \right\|_{\infty} |f_{\lambda}(t)|$. Note that since $\frac{\cosh \rho_{\lambda}}{\sinh \rho_{\lambda}} \geq 1$, we have
\begin{equation} \label{eq:bound_derivative}
\frac{1}{2 |f_{\lambda}(t)|} \cdot \frac{d}{dt} |f_{\lambda}(t)|^{2} \leq - \lambda |f_{\lambda}(t)| (1 - |f_{\lambda}(t)|) + \left\| \theta' \right\|_{\infty}
\end{equation}
Now let $C > \left\| \theta' \right\|_{\infty}$. Since $f_{\lambda}(0) = 0$, it is then clear that if $\lambda$ is large enough and $\lambda|f_{\lambda}(t)|$ reaches $C$, the right hand side in \eqref{eq:bound_derivative} will become negative and hence $|f_{\lambda}(t)|$ will be forced to decrease. Thus, we conclude that there exists a $C$ such that for all large $\lambda$, we have
\begin{align} \label{bound for first order difference}
\sup_{t \in [0,1]} |f_{\lambda}(t)| \leq \frac{C}{\lambda}. 
\end{align}

\begin{flushleft}
\textit{Step 2.} 
\end{flushleft}
We now show that the projection of $f_{\lambda}(t)$ onto the direction $\theta(t)$ is of order $\oO(\frac{1}{\lambda^{2}})$. Similar as before, we compute
\begin{align*}
\frac{d}{dt} |\big< \theta(t), f_{\lambda}(t) \big>|^{2} = 2 \big< \theta(t), f_{\lambda}(t) \big> \bigg( \big< \theta'(t), f_{\lambda}(t) \big> + \big< \theta(t), f_{\lambda}'(t) \big> \bigg). 
\end{align*}
For the second term in the parenthesis above, substituting $f_{\lambda}'(t)$ by \eqref{eq:equation_f}, we have
\begin{align*}
\big< \theta(t), f_{\lambda}'(t) \big> = - \big< \theta(t), \theta'(t) \big> - \lambda \cdot \coth \rho_{\lambda}(t) \bigg(  2 \big< \theta(t), f_{\lambda}(t) \big> + \big< \theta(t), f_{\lambda}(t) \big>^{2} \bigg). 
\end{align*}
Note that $|\theta| \equiv 1$, so $\big< \theta, \theta' \big> \equiv 0$, and thus we get
\begin{align*}
\frac{1}{2} \frac{d}{dt} |\big< \theta(t), f_{\lambda}(t) \big>|^{2} =  &\big< \theta(t), f_{\lambda}(t) \big> \bigg[ - \lambda \cdot \coth \rho_{\lambda}(t) \bigg(  2 \big< \theta(t), f_{\lambda}(t) \big> \\
&+ \big< \theta(t), f_{\lambda}(t) \big>^{2} \bigg) + \big< \theta'(t), f_{\lambda}(t) \big>  \bigg]. 
\end{align*}
Since $\sup_{t} |f_{\lambda}(t)| < \frac{\tilde{C}}{\lambda}$ by the first step, we have
\begin{equation} \label{eq:projection_derivative}
\sup_{t} |\big< \theta'(t), f_{\lambda}(t) \big>| \leq \frac{\tilde{C} \left\| \theta' \right\|_{\infty}}{\lambda}
\end{equation}
for some $\tilde{C} > 0$, uniformly in $\lambda$. Also, since $|f_{\lambda}| < \frac{C}{\lambda}$ by Step 1 and $\theta$ is a unit vector, we have
\begin{equation} \label{eq:intermediate}
|2 \big< \theta(t), f_{\lambda}(t) \big> + \big< \theta(t), f_{\lambda}(t) \big>^{2}| > | \big< \theta(t), f_{\lambda}(t) \big> | 
\end{equation}
for all large $\lambda$. Since $\frac{\cosh \rho_{\lambda}}{\sinh \rho_{\lambda}} \geq 1$, if $|\big< \theta(t), f_{\lambda}(t) \big>| > \frac{\tilde{C} \left\| \theta' \right\|_{\infty}}{{\lambda}^{2}}$ for the same $\tilde{C}$ as in \eqref{eq:projection_derivative} at some $t > 0$, by combining \eqref{eq:projection_derivative} and \eqref{eq:intermediate}, it is easy to see that we will have
\begin{align*}
\lambda \coth \rho_{\lambda}(t) \bigg| \bigg(  2 \big< \theta(t), f_{\lambda}(t) \big> + \big< \theta(t), f_{\lambda}(t) \big>^{2} \bigg) \bigg| > \lambda | \big< \theta(t), f_{\lambda}(t) \big> | > |\big< \theta'(t), f_{\lambda}(t) \big>|, 
\end{align*}
and the minus sign in front of this term will make $\frac{d}{dt} |\big< \theta(t), f_{\lambda}(t) \big>|^{2}$ negative, and hence the quantity $|\big< \theta(t), f_{\lambda}(t) \big>|$ will be forced to decrease. Therefore, there exists a $C > 0$ such that
\begin{align} \label{projection is even smaller}
\sup_{t \in [0,1]} |\big< \theta(t), f_{\lambda}(t) \big>| \leq \frac{C}{\lambda^{2}}
\end{align}
for all large $\lambda$.

\begin{flushleft}
\textit{Step 3.}
\end{flushleft}
We now show that for large $\lambda$, the quantity $\frac{\cosh \rho_{\lambda}(t)}{\sinh \rho_{\lambda}(t)} - 1$ is exponentially small at positive times. To see this, note that
\begin{align*}
\rho_{\lambda}'(s) = \lambda \big(1 + \big< \theta(s), f_{\lambda}(s)\big> \big). 
\end{align*}
Integrating both sides from $0$ to $t$, and employing \eqref{projection is even smaller}, we have
\begin{equation} \label{recovering the length via differential equation}
\lambda t - \frac{Ct}{\lambda} \leq \rho_{\lambda}(t) \leq \lambda t
\end{equation}
for all $t$, where the second inequality follows from the geometry that geodesic development gives the maximal possible length $\lambda t$. Now, since
\begin{equation} \label{eq:length_exponential}
\coth \rho_{\lambda}(t) - 1 = \frac{2 e^{- 2\rho_{\lambda}(t)}}{1 - e^{- 2 \rho_{\lambda}(t)}}, 
\end{equation}
\eqref{recovering the length via differential equation} implies that this quantity is exponentially small in $\lambda t$ for positive $t$, but has a singularity at $t = 0$, which has size $\frac{1}{\lambda t}$. On the other hand, since $f_{\lambda}(0) = 0$ and $|f_{\lambda}'|$ is bounded, this singularity can be killed by a multiplication of $f_{\lambda}(t)$. Thus, we have
\begin{align} \label{exponential decay of the ratio}
\big( \coth \rho_{\lambda}(t) - 1 \big) |f_{\lambda}(t)| \leq C e^{- \lambda t}, 
\end{align}
where $C$ is independent of $\lambda$ and $t$.

\begin{flushleft}
\textit{Step 4.} 
\end{flushleft}
We are now ready to prove the main claim. For convenience, write $P(s) = I + \theta(s) \theta(s)^{T}$, and we can rewrite \eqref{eq:equation_f} as
\begin{align*}
f_{\lambda}'(s) = - \lambda P(s) f_{\lambda}(s) - \theta'(s) + r_{\lambda}(s), 
\end{align*}
where
\begin{align*}
r_{\lambda}(s) = - \lambda \big( \coth \rho_{\lambda}(t) - 1 \big) \big( P(s) + f_{\lambda}(s) \theta(s)^{T} \big) f_{\lambda}(s) - \lambda f_{\lambda}(s) \theta(s)^{T} f_{\lambda}(s). 
\end{align*}
Since by Steps $2$ and $3$, we have $|f_{\lambda}(s) \theta(s)^{T} f_{\lambda}(s)| < \frac{C}{\lambda^{3}}$ and $\big( \frac{\cosh \rho_{\lambda}(t)}{\sinh \rho_{\lambda}(t)} - 1 \big) |f_{\lambda}(t)| \leq C e^{- \lambda t}$, it follows that
\begin{align*}
\sup_{s} |r_{\lambda}(s)| < \frac{C}{\lambda^{2}}. 
\end{align*}
This suggests that the nonlinear part of the equation is ``small" when $\lambda$ is large. 

Now, fix an arbitrary $t > 0$, and we want to compute the large $\lambda$ limit of $\lambda f_{\lambda}(t)$. Since $\gamma \in \cC^{2}$ at natural parametrization, it follows that $|\theta'|$ is bounded, and thus there exists $\kappa > 0$ such that $\forall \epsilon > 0$ and all $s \in [t - \kappa \epsilon, t]$, we have
\begin{align*}
|P(s) - P(t)| < \epsilon. 
\end{align*}
As a consequence, we have
\begin{align*}
f_{\lambda}'(s) = - \lambda P(t) f_{\lambda}(s) - \theta'(s) + r_{\lambda}(s) + \tilde{r}_{\lambda}(s), 
\end{align*}
where
\begin{align*}
\tilde{r}_{\lambda}(s) = - \lambda \big(P(s) - P(t)\big) f_{\lambda}(s)
\end{align*}
satisfies $|\tilde{r}_{\lambda}(s)| < C \epsilon$ for all $s \in [t - \kappa \epsilon, t]$. Now, if $g_{\lambda}$ satisfies the equation
\begin{align*}
g_{\lambda}'(s) = - \lambda P(t) g_{\lambda}(s) - \theta'(s), \qquad s \in [t - \kappa \epsilon, t]
\end{align*}
with initial condition $g_{\lambda}(t - \kappa \epsilon) = f_{\lambda}(t - \kappa \epsilon)$, by the bounds on $|r_{\lambda}|$ and $|\tilde{r}_{\lambda}|$, we see that
\begin{equation} \label{eq:difference_f_g}
|f_{\lambda}(t) - g_{\lambda}(t)| < C \epsilon \big( \frac{1}{\lambda^{2}} + \epsilon \big). 
\end{equation}
Note that the equation defining $g_{\lambda}$ is linear with constant coefficient, so the terminal value $g_{\lambda}(t)$ can be expressed explicitly by
\begin{equation} \label{eq:constant_solution}
g_{\lambda}(t) = e^{- \lambda \kappa \epsilon P(t)} f_{\lambda}(t - \kappa \epsilon) - e^{- \lambda t P(t)} \bigg( \int_{t - \kappa \epsilon}^{t} e^{\lambda s P(t)} ds \bigg) \theta'(t) + \text{Error}, 
\end{equation}
where the error term is given by
\begin{align*}
\text{Error} =  e^{-\lambda t P(t)}  \int_{t - \kappa \epsilon}^{t} e^{\lambda s P(t)} (\theta'(t) - \theta'(s)) ds. 
\end{align*}
For the second term on the right hand side in \eqref{eq:constant_solution}, we have
\begin{equation} \label{eq:second_term}
e^{- \lambda t P(t)} \int_{t - \kappa \epsilon}^{t} e^{\lambda s P(t)} \theta'(t) ds = \frac{1}{\lambda} P(t)^{-1} \big( I - e^{- \lambda \kappa \epsilon P(t)} \big) \theta'(t). 
\end{equation}
Similarly, for the error term, we have
\begin{equation} \label{eq:error}
\bigg|e^{-\lambda t P(t)} \int_{t - \kappa \epsilon}^{t} e^{\lambda s P(t)} (\theta'(t) - \theta'(s) ) ds \bigg| \leq \frac{C}{\lambda} \delta(\epsilon) \big| P(t)^{-1} \big( I - e^{- \lambda \kappa \epsilon P(t)} \big) \big|, 
\end{equation}
where 
\begin{align*}
\delta(\epsilon) = \sup_{s \in [t - \kappa \epsilon, t]} \big|\theta'(s) - \theta'(t) \big| \rightarrow 0
\end{align*}
as $\epsilon \rightarrow 0$. Since $P(t)$ is symmetric and positive definite with eigenvalues $2$ and $1$, it has a uniformly bounded inverse, so the error term can then be bounded by
\begin{equation} \label{eq:error_bound}
\text{Error} < \frac{C}{\lambda} \delta(\epsilon). 
\end{equation}
Also, for the first term on the right hand side of \eqref{eq:constant_solution}, taking $\epsilon = \lambda^{-\frac{2}{3}}$, we have
\begin{equation} \label{eq:first_bound}
\big| e^{- \lambda \kappa \epsilon P(t)} f_{\lambda}(t-\kappa \epsilon) \big| < e^{- c \lambda^{\frac{1}{3}}}. 
\end{equation}
Thus, multiplying both sides of \eqref{eq:constant_solution} by $\lambda$, taking $\epsilon =  \lambda^{-\frac{2}{3}}$, and sending $\lambda$ to $+\infty$, we deduce from \eqref{eq:difference_f_g} \eqref{eq:second_term}, \eqref{eq:error_bound} and \eqref{eq:first_bound} that
\begin{align*}
\lim_{\lambda \rightarrow +\infty} \lambda f_{\lambda}(t) = - P(t)^{-1} \theta'(t). 
\end{align*}
This finishes the proof. 
\end{proof}

\begin{rmk}
We should note that Step 2 in the proof above is not necessary in establishing \eqref{exponential decay of the ratio}. Without that step, the term $\frac{1}{\lambda^{2}}$ on the right hand side of \eqref{eq:difference_f_g} will become $\frac{1}{\lambda}$, which would not affect the conclusion of the theorem. However, the estimate in Step 2 gives a sharp estimate \eqref{recovering the length via differential equation}, which in turn verifies \eqref{eq:recovery_length} (\cite[Prop.~3.8]{HL10}) in the case of $\cC^{2}$ paths. 
\end{rmk}

\subsection{Higher order derivatives}

Since we know how to recover the length from the signature, from now on, we assume without loss of generality that $L=1$. In the previous subsection, we showed that we can recover the derivative at the end point of $\gamma$ through a limiting process by
\begin{align*}
\theta(1) = \lim_{\lambda \rightarrow +\infty} \eta_{\lambda}(1), 
\end{align*}
where the right hand side can be ``observed" on the hyperboloid for each $\lambda$. By Theorem \ref{th:recovery_derivative}, we have
\begin{equation} \label{eq:second_derivative}
\theta'(1) = \lim_{\lambda \rightarrow +\infty} \lambda (I + \theta(1) \theta(1)^{T}) (\theta(1) - \eta_{\lambda}(1)). 
\end{equation}
Since $\theta(1)$ is now known, \eqref{eq:second_derivative} gives the value of $\theta'(1)$ through another limiting process. This suggests that we can actually recover higher order derivatives at time $t=1$ provided $\gamma$ is sufficiently smooth. 

We first give a heuristic argument to see how it works. By the estimates \eqref{recovering the length via differential equation} and \eqref{eq:length_exponential} on $\coth \rho_{\lambda}(t) - 1$, we can write the equation for $f_{\lambda}$ as
\begin{equation} \label{eq:simple_equation_f}
f_{\lambda}'(t) = - \lambda P(t) f_{\lambda}(t) - \lambda f_{\lambda}(t) \theta(t)^{T} f_{\lambda}(t) - \theta'(t) + r_{\lambda}(t), 
\end{equation}
where $P(t) = I + \theta(t) \theta(t)^{T}$, and the remainder $r_{\lambda}$ satisfies $|r_{\lambda}(t)| < C e^{- \lambda t}$ uniformly in $\lambda$ and $t \geq \tau$ for any fixed positive $\tau$. 

Let us also assume for a moment that for any $t > 0$, $f_{\lambda}(t)$ can be expanded around $\lambda = +\infty$ by
\begin{equation} \label{eq:expansion_f}
f_{\lambda}(t) = \frac{A_{1}(t)}{\lambda} + \cdots + \frac{A_{n}(t)}{\lambda^{n}} + \cdots. 
\end{equation}
Under suitable regularity conditions of $\gamma$, we can also differentiate the above series term-wise to get
\begin{align*}
f_{\lambda}'(t) = \frac{A_{1}'(t)}{\lambda} + \cdots + \frac{A_{n}'(t)}{\lambda^{n}} + \cdots. 
\end{align*}
Now, substituting the expansions of $f_{\lambda}$ and $f_{\lambda}'$ into \eqref{eq:simple_equation_f}, and comparing coefficients of $\frac{1}{\lambda^{n}}$ on both sides, we get $A_{1}(t) = - P(t)^{-1} \theta'(t)$, and
\begin{equation} \label{eq:recursive_relation}
A_{n+1}(t) = - P(t)^{-1} \bigg( A_{n}'(t) + \sum_{j=1}^{n} A_{j}(t) \theta(t)^{T} A_{n+1-j}(t) \bigg)
\end{equation}
for $n = 1, 2, \dots$. Note that this definition is consistent with $n=0$ if we set $A_{0}(t) = \theta(t)$. It is clear that if $\theta \in \cC^{k}$ (or $\gamma \in \cC^{k+1}$), then $A_{n}$ can be defined with the above recursive relation up to $n = k$ with $A_{n} \in \cC^{k-n}$. We now show that $f_{\lambda}$ does have an expansion as in \eqref{eq:expansion_f}.

\begin{thm} \label{higher order derivatives}
Let $\gamma$ be a $\cC^{k+1}$ path at natural parametrization and has length $1$, and $\theta = \dot{\gamma}$. Let $f_{\lambda}(t)$ be the solution to the differential equation \eqref{eq:equation_f}. Let $A_{1}(t) = - P(t)^{-1} \theta'(t)$, and $A_{n}(t)$ be defined with the recursive relation \eqref{eq:recursive_relation} for $n = 1, \dots, k$. Then, we have
\begin{align*}
\lim_{\lambda \rightarrow +\infty} \lambda^{n+1} \bigg(  f_{\lambda}(t) - \sum_{j=1}^{n} \frac{A_{j}(t)}{\lambda^{j}} \bigg) = A_{n+1}(t)
\end{align*}
for each $n = 0, 1, \dots, k-1$. 
\end{thm}
\begin{proof}
The theorem is clearly true for $n=0$. For $1 \leq n \leq k-1$, let
\begin{equation} \label{eq:expression_g}
g_{\lambda}(t) = \lambda^{n} \bigg( f_{\lambda}(t) - \sum_{j=1}^{n} \frac{A_{j}(t)}{\lambda^{j}} \bigg), 
\end{equation}
and we want to show that $\lambda g_{\lambda}(t) \rightarrow A_{n+1}(t)$. The main step is to derive a differential equation for $g_{\lambda}$ that is comparable to the one for $f_{\lambda}$ as in \eqref{eq:equation_f}. First, we see from \eqref{eq:expression_g} that
\begin{align*}
f_{\lambda}(t) = \frac{g_{\lambda}(t)}{\lambda^{n}} 
= \sum_{j=1}^{n} \frac{A_{j}(t)}{\lambda^{j}}, \qquad f_{\lambda}'(t) = \frac{g_{\lambda}'(t)}{\lambda^{m}} - \sum_{j=1}^{m} \frac{A_{j}'(t)}{\lambda^{j}}. 
\end{align*}
Here, $A_{j} \in \cC^{k-j}$, so it is differentiable up to $j = k-1$. Now, substituting these two expressions into \eqref{eq:equation_f}, multiplying $\lambda^{n}$ on both sides, and employing the recursive relation \eqref{eq:recursive_relation}, we derive that
\begin{equation} \label{eq:equation_g}
g_{\lambda}'(t) = - \lambda P(t) g_{\lambda}(t) + P(t) A_{n+1}(t) + r_{\lambda}(t) + r_{\lambda}^{(1)}(t) + r_{\lambda}^{(2)}(t), 
\end{equation}
where $r_{\lambda}$ is the same as in \eqref{eq:simple_equation_f}, and
\begin{align*}
r_{\lambda}^{(1)}(t) = - \frac{1}{\lambda^{n-1}} \cdot g_{\lambda}(t) \theta(t)^{T} g_{\lambda}(t) - \sum_{j=1}^{n} \frac{1}{\lambda^{j-1}} \bigg( A_{j}(t) \theta(t)^{T} g_{\lambda}(t) + g_{\lambda}(t) \theta(t)^{T} A_{j}(t) \bigg), 
\end{align*}
and
\begin{align*}
r_{\lambda}^{(2)}(t) = \sum_{j+\ell \geq n+2} \frac{1}{\lambda^{j+\ell-n-1}} \cdot A_{j}(t) \theta(t)^{T} A_{\ell}(t). 
\end{align*}
Now, note that
\begin{align*}
\sup_{t} \big( |r_{\lambda}(t)| + |r_{\lambda}^{(2)}(t)| \big) < \frac{C}{e^{- \lambda t} + \frac{1}{\lambda}}. 
\end{align*}
Proceeding with exactly the same way as in Theorem \ref{th:recovery_derivative}, we can show that
\begin{align*}
\sup_{t} |r_{\lambda}^{(1)}(t)| < \frac{C}{\lambda}, 
\end{align*}
and
\begin{align*}
\lim_{\lambda \rightarrow +\infty} \lambda g_{\lambda}(t) = P(t)^{-1} P(t) A_{n+1}(t) = A_{n+1}(t). 
\end{align*}
This completes the proof of the theorem. 
\end{proof}

We have the following immediate consequence on recovering $\theta^{(n)}(1)$ from the signature.

\begin{cor}
Let $\gamma$ be a $\cC^{k+1}$ ($k \geq 1$) path at natural parametrization. Then, all the $(k+1)$ derivatives at the end point of $\gamma$ can be recovered from its signature sequence. 
\end{cor}
\begin{proof}
For each $\lambda$, $\eta_{\lambda}(1)$ is an observable from the signature of $\gamma$. By Theorem \ref{higher order derivatives}, we have the expansion
\begin{equation} \label{eq:expansion}
\eta_{\lambda}(1) = A_{0} + \frac{A_{1}}{\lambda} + \cdots + \frac{A_{k}}{\lambda^{k}} + o(\lambda^{-k}), 
\end{equation}
where $A_{j} = A_{j}(1)$ satisfies the recursive relation \eqref{eq:recursive_relation}, and $A_{0} = A_{0}(1) = \theta(1)$. \eqref{eq:expansion} implies that all $A_{j}$'s up to $j=k$ can be observed from the signature. It then remains to show that each $\theta^{(j)}(1)$ can be explicitly expressed in terms of $A_{0}, \dots, A_{j}$ for all $j \leq k$. 

The case $j=0$ is immediate since $\theta(1) = A_{0}$. Suppose for some $n \geq 0$, $\theta^{(n)}(1)$ can be explicitly expressed in terms of $A_{0}, A_{1}, \dots, A_{n}$, then $\theta^{(n+1)}(1)$ can be explicitly written with $A_{0}, \dots, A_{n}$ and $A_{0}', \dots, A_{n}'$, where $A_{j}' = A_{j}'(1)$. But by the recursive relation \eqref{eq:recursive_relation}, each $A_{j}'$ can be written as an explicit function of $A_{0}, \dots, A_{j+1}$. This implies that $\theta^{(n+1)}(1)$ can be expressed explicitly in terms of $A_{0}, \dots, A_{n+1}$. This completes the proof. 
\end{proof}

\section{Inversion for piecewise linear paths}

We now use the construction in Section \ref{sec:hyperbolic} to recover piecewise linear paths from their signatures. We will show that, for such paths, both the direction and the length of the last linear piece can be recovered by explicitly writing down the matrices of Cartan development of the path (as in Example \ref{eg:lattice}). Thus, we can remove the last linear piece from the whole signature. Applying this procedure repeatedly gives the path back from its signature. 

More precisely, we will show below that if the length of the last piece is $l$, then we have
\begin{align*}
c e^{- \lambda l} < |\eta_{\lambda}(1) - \theta(1)| < C e^{- \lambda l}
\end{align*}
for all large $\lambda$, where $\theta(1)$ is the direction of the path at terminal time $t=1$. Thus, one can recover $l$ from the asymptotic behavior of $|\eta_{\lambda}(1) - \theta(1)|$. This will be an immediate consequence of the following theorem. 

\begin{thm} \label{th:piecewise_linear}
	Let $\gamma = \alpha_{1} * \cdots \alpha_{n}$ be a piecewise linear path, where each $\alpha_{j}$ is a linear piece with direction $\theta_{j} \in \mathbb{S}^{d-1}$ and length $l_{j}$, and that $\theta_{j} \neq \pm \theta_{j+1}$ for all $j$. Let $L_{j} = l_{1} + \cdots + l_{j}$, and $\gamma_{j} = \alpha_{1} * \cdots * \alpha_{j}$. For each $\lambda > 0$, let
	\begin{align*}
	\begin{pmatrix} \eta_{\lambda,j} \sinh \rho_{\lambda,j} \\ \cosh \rho_{\lambda,j}  \end{pmatrix}
	\end{align*}
	be the end point of the development of the rescaled path $\gamma_{\lambda,j} = \lambda \gamma_{j}$ on the hyperboloid. Then, there exists $c, C > 0$ such that
	\begin{align} \label{induction}
	\lambda L_{j} - C < \rho_{\lambda,j} \leq \lambda L_{j}, \qquad c e^{-\lambda l_{j}} < |\eta_{\lambda,j} - \theta_{j}| < C e^{-\lambda l_{j}}
	\end{align}
	for all $1 \leq j \leq n$ and all large enough $\lambda$. 
\end{thm}
\begin{proof}
	The theorem is true for $j=1$, since $\gamma_{1} = \alpha_{1}$ is a straight line. Suppose \eqref{induction} is true for $j = 1, \dots ,k-1$, and we need to prove it for $j=k$. By symmetry of $\HH^{d}$, we can assume without loss of generality that $\theta_{k} = e_{1} = (1, 0, \dots, 0)^{T}$, so the matrix of the Cartan development of $\lambda \alpha_{k}$ is given by
	\begin{align*}
	\begin{pmatrix} \cosh \lambda l_{k} & \phantom{1} & \phantom{1} & \phantom{1} & \sinh \lambda l_{k}  \\   \phantom{1} & 1 & \phantom{1} & \phantom{1} & \phantom{1} \\  \phantom{1} & \phantom{1} & \ddots & \phantom{1} & \phantom{1} \\ \phantom{1} & \phantom{1} & \phantom{1} & 1 & \phantom{1}  \\  \sinh \lambda l_{k}  & \phantom{1} & \phantom{1} & \phantom{1} & \cosh \lambda l_{k}  \end{pmatrix}. 
	\end{align*}
	Multiplying this matrix to the end point of the development of $\lambda \gamma_{k-1}$ on $\HH^{d}$, we then obtain the end point of $\lambda \gamma_{k}$ on $\HH^{d}$ to be
	\begin{equation} \label{eq:end_point}
	\begin{pmatrix}  \eta_{\lambda,k}^{(1)} \sinh \rho_{\lambda,k} \\ \eta_{\lambda,k}^{(2)} \sinh \rho_{\lambda,k} \\ \vdots \\ \eta_{\lambda,k}^{(d)} \sinh \rho_{\lambda,k} \\ \cosh \rho_{\lambda,k}  \end{pmatrix} = \begin{pmatrix}  \eta_{\lambda,k-1}^{(1)} \cosh \lambda l_{k} \sinh \rho_{\lambda,k-1} + \sinh \lambda l_{k} \cosh \rho_{\lambda,k-1} \\ \eta_{\lambda,k-1}^{(2)} \sinh \rho_{\lambda,k-1} \\ \vdots \\ \eta_{\lambda,k-1}^{(d)} \sinh \rho_{\lambda,k-1} \\ \eta_{\lambda,k-1}^{(1)} \sinh \lambda l_{k} \sinh \rho_{\lambda,k-1} + \cosh \lambda l_{k} \cosh \rho_{\lambda,k-1}  \end{pmatrix}, 
	\end{equation}
	where $\eta_{\lambda,k} = (\eta_{\lambda,k}^{(1)}, \dots, \eta_{\lambda,k}^{(d)})^{T}$. Since $\theta_{k} = e_{1}$, the assumption $\theta_{k} \neq \pm \theta_{k+1}$ and the induction hypothesis on $|\eta_{\lambda,k-1} - \theta_{k-1}|$ implies that there exists $\delta > 0$ such that
	\begin{equation} \label{eq:direction_induction}
	-1+\delta < \eta_{\lambda,k-1}^{(1)} < 1 - \delta
	\end{equation}
	for all large $\lambda$. Combining \eqref{eq:direction_induction} and the identity 
	\begin{align*}
	\cosh \rho_{\lambda,k} = \eta_{\lambda,k-1}^{(1)} \sinh \lambda l_{k} \sinh \rho_{\lambda,k-1} + \cosh \lambda l_{k} \cosh \rho_{\lambda,k-1}
	\end{align*}
	in the last row of \eqref{eq:end_point}, we deduce that there exists $c > 0$ such that
	\begin{align*}
	c \cosh (\rho_{\lambda,k-1} + \lambda l_{k}) < \cosh \rho_{\lambda,k} \leq \cosh (\rho_{\lambda,k-1} + \lambda l_{k}). 
	\end{align*}
	Taking logarithm on both sides, and using the induction hypothesis on $\rho_{\lambda,k-1}$, we conclude that
	\begin{equation} \label{eq:length_induction}
	\lambda L_{k} - C < \rho_{\lambda,k} \leq \lambda L_{k}. 
	\end{equation}
	We now turn to the direction $\eta_{\lambda,k}$. For $j \geq 2$, we have
	\begin{align*}
	\eta_{\lambda,k}^{(j)} = \eta_{\lambda,k-1}^{(j)} \cdot \frac{\sinh \rho_{\lambda,k-1}}{\sinh \rho_{\lambda,k}}, 
	\end{align*}
	which, combined with the hypothesis \eqref{induction} and the bound \eqref{eq:length_induction}, gives
	\begin{align} \label{eq:rest_directions}
	c \bigg( \sum_{j=2}^{d} |\eta_{\lambda,k-1}^{(j)}|^{2} \bigg) e^{- 2 \lambda l_{k}} < \sum_{j=2}^{d} |\eta_{\lambda,k}^{(j)}|^{2} < C \bigg( \sum_{j=2}^{d} |\eta_{\lambda,k-1}^{(j)}|^{2} \bigg) e^{- 2 \lambda l_{k}}
	\end{align}
	for some $c, C > 0$, uniformly over all large $\lambda$. Also, \eqref{eq:direction_induction} implies that $\sum_{j=2}^{d} |\eta_{\lambda,k-1}^{(j)}|^{2}$ is bounded away from $0$ uniformly in $\lambda$, so \eqref{eq:rest_directions} becomes
	\begin{equation} \label{rest_directions}
	c e^{- 2 \lambda l_{k}} < \sum_{j=2}^{d} |\eta_{\lambda,k}^{(j)}|^{2} < C e^{- 2 \lambda l_{k}}. 
	\end{equation}
	As for $\eta_{\lambda,k}^{(1)}$, we first note that by the identity
	\begin{align*}
	\eta_{\lambda,k}^{(1)} \sinh \rho_{\lambda,k} = \eta_{\lambda,k-1}^{(1)} \cosh \lambda l_{k} \sinh \rho_{\lambda,k-1} + \sinh \lambda l_{k} \cosh \rho_{\lambda,k-1}
	\end{align*}
	in the first row of \eqref{eq:end_point}, the induction hypothesis \eqref{induction} on $\rho_{\lambda,k-1}$ and the bound \eqref{eq:length_induction}, we necessarily have $\eta_{\lambda,k}^{(1)} \rightarrow 1$ as $\lambda \rightarrow +\infty$. On the other hand, since
	\begin{align*}
	1 - |\eta_{\lambda,k}^{(1)}|^{2} = \sum_{j=2}^{d} |\eta_{\lambda,k}^{(j)}|^{2}, 
	\end{align*}
	the bound \eqref{rest_directions} as well as the fact that $|\eta_{\lambda,k}^{(1)}| < 1$ imply that
	\begin{equation} \label{eq:first_direction}
	c e^{- 4 \lambda l_{k}} < \frac{1}{9} \big(1 - |\eta_{\lambda,k}^{(1)}|^{2} \big)^{2} \leq |1 - \eta_{\lambda,k}^{(1)}|^{2} \leq \big(1 - |\eta_{\lambda,k}^{(1)}|^{2} \big)^{2} < C e^{-4 \lambda l_{k}}, 
	\end{equation}
	where the constants $c,C$ are independent of $\lambda$. Combining \eqref{rest_directions} and \eqref{eq:first_direction}, we then obtain
	\begin{align*}
	c e^{- \lambda l_{k}} < |\eta_{\lambda,k} - \theta_{k}| < C e^{-  \lambda l_{k}}. 
	\end{align*}
	This finishes the induction and thus the proof of the theorem. 
\end{proof}

\begin{cor}
	As an immediate corollary, we have
	\begin{align} \label{recover last piece}
	\lim_{\lambda \rightarrow +\infty} \frac{1}{\lambda} \log |\eta_{\lambda,n} - \theta_{n}| = - l. 
	\end{align}
	Since $\theta_{n}$ can be first recovered through the limiting process
	\begin{align*}
	\theta_{n} = \lim_{\tilde{\lambda} \rightarrow +\infty} \eta_{\tilde{\lambda},n}, 
	\end{align*}
	\eqref{recover last piece} indeed recovers the length of the last linear piece from the signature. 
\end{cor}

\begin{rmk}
	Note that the proof of Theorem \ref{th:piecewise_linear} uses the construction in Section \ref{sec:hyperbolic} only. In particular, it does not use the conclusion of Theorem \ref{th:recovery_derivative}. 
\end{rmk}

\begin{rmk}
In view of \eqref{eq:expansion}, it is not surprising that for piecewise linear paths, the difference $|\eta_{\lambda}(1) - \theta(1)|$ is exponentially small in $\lambda$. This is because for piecewise linear paths, we have $\theta^{(j)}(1) = 0$, and hence also $A_{j}(1) = 0$ for all $j \geq 1$. 
\end{rmk}

\bigskip

\textsc{Mathematical and Oxford-Man Institutes, University of Oxford, Woodstock road, Oxford, OX2 6GG, UK}. 

Email: tlyons@maths.ox.ac.uk

\smallskip

\textsc{Mathematics Institute, University of Warwick, Coventry, CV4 7AL, UK}. 

Email: weijun.xu@warwick.ac.uk

\end{document}